\documentclass[11pt]{amsart}

\usepackage[utf8]{inputenc}
\usepackage[backend=biber,style=alphabetic,maxbibnames=50]{biblatex}
\usepackage{amssymb,amsfonts}
\usepackage{amsmath,amscd}
\usepackage[all,arc]{xy}
\usepackage{enumerate}
\usepackage{mathrsfs}
\usepackage[pdftex]{graphicx}
\usepackage[T1]{fontenc}
\usepackage[usenames,dvipsnames]{color}
\usepackage[bookmarks=false]{hyperref}
\usepackage{dsfont}
\usepackage{tikz-cd}
\usetikzlibrary{automata,positioning}
\usepackage{fullpage}
\usepackage{wasysym}

\theoremstyle{plain}
\newtheorem{thm}{Theorem}[section]

\newtheorem{prop}[thm]{Proposition}
\newtheorem{lem}[thm]{Lemma}

\theoremstyle{definition}

\newtheorem{aDD^+m}[thm]{ADD^+endum}

\theoremstyle{remark}
\newtheorem{rmk}[thm]{Remark}

\newcommand*{\defeq}{\mathrel{\vcenter{\baselineskip0.5ex \lineskiplimit0pt
			\hbox{\scriptsize.}\hbox{\scriptsize.}}}%
	=}

\DeclareMathOperator{\VD}{VD}

\DeclareMathOperator{\Jac}{Jac}

\DeclareMathOperator{\Leb}{Leb}

\title{A Ruelle-McMullen formula for the
 volume dimension\\
  of skew products in $\mathbb C^2$}

\author{Fabrizio Bianchi}
\address{Dipartimento di Matematica, Università di Pisa, Largo Bruno Pontecorvo 5, 56127 Pisa, Italy}
 \email{fabrizio.bianchi$@$unipi.it}
\author{Yan Mary He}
\address{Department of Mathematics\\
	University of Oklahoma\\
	Norman, OK 73019}
\email{he$@$ou.edu}
\date{\today}

\begin{document}

\begin{abstract}
Ruelle gave an explicit second-order expansion at $c=0$
of the Hausdorff dimension of the Julia set of the quadratic family
$f_c(z)=z^2+c$.
McMullen later extended this result to polynomial perturbations of
$z^d$
 for 
arbitrary degree $d\geq 2$.
In this paper we study an analogue of this problem for
 skew products in
$\mathbb C^2$.
Since holomorphic dynamical systems in higher dimensions are non-conformal, we replace the
Hausdorff dimension by the \emph{volume dimension}, a dynamically defined notion
we introduced in our earlier work and characterized as the zero of a natural pressure function.
We consider families of holomorphic skew products of the form
\[
f_t(z,w)=(z^d,
 w^d+t(c_1 (z) w^{d-1} +c_2(z)w^{d-2} + \cdots+c_d(z))).
\]
Our main result gives an explicit second-order expansion of the volume dimension of the
Julia set $J(f_t)$ as $t\to0$
 in terms of the coefficients $c_k(z)$.
 \end{abstract}

\maketitle

\section{Introduction}

In the early 1980s, Ruelle \cite{Ruelle82} initiated the study of the variation of the Hausdorff
dimension of Julia sets in holomorphic families. For the quadratic family
$f_c(z)=z^2+c$, he proved that the Hausdorff dimension of the Julia set admits the
expansion
\[
\mathrm{H.dim}(J_c)=1+\frac{|c|^2}{4\log 2}+O(|c|^3), \qquad |c|\to0.
\]
In particular, the Hausdorff dimension function has a strict local minimum at the
monomial map $z\mapsto z^2$.
McMullen
\cite{McMullen08}
 generalized Ruelle’s result to polynomials of arbitrary degree, obtaining
an explicit formula for the second derivative of the Hausdorff dimension at the monomial
$z\mapsto z^d$. Namely, 
he showed that the Hausdorff dimension of the Julia set of the map
$f_t(z) = z^d +t(c_2z^{d-2}+c_3z^{d-3}+\cdots+c_d)$
satisfies
\[ \mathrm{H.dim} (J_t) = 1 + \frac{ |t|^2}{4d^2\log d}
\sum_{k=2}^dk^2|c_k|^2 + O(|t|^3), \qquad |t|\to0.\] 

The purpose of this paper is to investigate an analogue of this problem for
polynomial-like 
skew products in
$\mathbb C^2$.
More precisely, for $t$ near $0$ we consider families of the form
\begin{equation}\label{eq:form-sk}
f_t(z,w)=(z^d,\; w^d+t(c_1 (z) w^{d-1} + c_2(z)w^{d-2}+\cdots+c_d(z))),
\end{equation}
where the $c_j$'s are holomorphic functions in $z$,
and we study the variation of the dimension of their Julia sets. Examples of these maps
are the
regular polynomial skew products,
 i.e., polynomial endomorphisms
 of $\mathbb C^2$ of the form \eqref{eq:form-sk} 
 which extend 
 holomorphically to
 endomorphisms
of $\mathbb P^2(\mathbb C)$, see \cite{BedfordJonsson,Jonsson99}.

Contrary to the one-dimensional case, in higher-dimensional holomorphic dynamics
 the Hausdorff dimension is no
longer well adapted to the dynamics. 
Holomorphic maps in dimension greater than one are
non-conformal, and basic tools such as Koebe’s distortion theorem are no longer available.
As a consequence, the Hausdorff dimension may fail to be dynamically meaningful, and its
behaviour in parameter space is poorly understood.

To address this issue, in an earlier work \cite{BHMane} we introduced the notion of
\emph{volume dimension} for invariant sets and measures in higher-dimensional holomorphic
dynamical systems. The volume dimension coincides 
(up to a factor  of $2$)
with the Hausdorff dimension in complex
dimension one, but incorporates the non-conformal geometry of higher-dimensional systems.
In particular, it satisfies a Mañé--Manning type formula relating volume dimension,
entropy, and Lyapunov exponents, and can be characterized as the zero of a natural
pressure function for expanding invariant measures. For hyperbolic maps, this zero
coincides 
(up to the same factor of $2$)
with the volume dimension of the Julia set.

In this paper we study the real-analytic function
\[
t \longmapsto \VD_{f_t}(J(f_t)),
\]
where $J(f_t)$ denotes the Julia set of $f_t$.
Our main result 
gives
an explicit second-order expansion
of the volume dimension of $J(f_t)$
at $t=0$.

\begin{thm}\label{thm_main_sk}
Let $f_t$ be the family as in \eqref{eq:form-sk}.
Then
\[
\VD_{f_t}(J(f_t))
=
\frac12
+
\frac{|t|^2}{16 d^2\log d}
\int_{S^1}\sum_{k=1}^d k^2|c_k(z)|^2\,d\Leb_1(z)
+ O(|t|^3),
\qquad |t|\to0
\]
where $\Leb_1$ denotes the normalized 
Lebesgue
probability  measure on $S^1$.
\end{thm}

\subsection*{Main ingredients and structure of the proof}

The proof follows the general strategy of
McMullen, adapted to the
non-conformal setting of polynomial skew products.

The starting point is the characterization of the volume dimension as the zero of a
pressure function. Writing $\delta_t$ for this zero, the problem reduces to computing
the second derivative of $\delta_t$ at $t=0$.
A general second-derivative formula allows  one to
 express
  $\ddot\delta_0$ in terms of the variance of
the infinitesimal deformation $\dot \phi_0$
of the geometric potential $\phi_t = -\log |{\Jac} f_t|$, see 
Proposition
 \ref{thm_delta_second_derivative}.

A key observation, inspired by McMullen’s work
\cite{McMullen08}, is that this infinitesimal deformation
$\dot \phi_0$ 
is
not an arbitrary H\"older function, but a \emph{virtual coboundary}: it
is the trace on the Julia set 
of a
dynamically defined H\"older function $h$
 on the basin of infinity, which
 can be written as
the difference of a function $g$ 
and its pull-back by the dynamics. 
This structure allows us to express the variance we need
in terms of the
asymptotic growth of the function $g$ near the boundary of the basin of infinity, see 
Proposition \ref{thm_main_2}.
A direct computation of this behaviour, see Proposition \ref{prop_extra_term},
 then completes the proof of Theorem \ref{thm_main_sk}.

\subsection*{Acknowledgements}
This project has received funding from
 the
 Programme
 Investissement d'Avenir
(ANR QuaSiDy /ANR-21-CE40-0016,
ANR PADAWAN /ANR-21-CE40-0012-01,
ANR TIGerS /ANR-24-CE40-3604),
 from the MIUR Excellence Department Project awarded to the Department of Mathematics of the University of Pisa, CUP I57G22000700001.
The first
author is affiliated to the GNSAGA group of INdAM.

\section{Preliminaries}\label{sec_prel}

\subsection{Holomorphic skew products}\label{sec_fibered}

Throughout this paper, we work with holomorphic skew products 
$F\colon\mathbb C^2 \to \mathbb C^2 $
of the form
\begin{equation}\label{eq:F-skew}
F(z,w)=(z^d,q(z,w)) = \Big(z^d, w^d + \sum_{k=1}^d c_k(z)w^{d-k}\Big),
\end{equation}
where $d\geq 2$
and each $c_k(z)$ is holomorphic in a neighbourhood of the closed unit disc
$\overline{\mathbb D}$\footnote{All the results of this section are still valid when $z^d$ is replaced by a hyperbolic polynomial $p$, not necessarily of degree $d$. We 
just describe the case $p(z)=z^d$ for simplicity and since we will only need this case in the sequel.}.
We denote the set of holomorphic skew products as above by $\mathcal S(d)$.
Restricting any $F\in \mathcal S(d)$
 to $S^1\times\mathbb C$,
 we obtain a \emph{fibered polynomial map} over
 the map
 $z\mapsto z^d$, namely
 a continuous map
\[
f\colon S^1\times\mathbb C\longrightarrow S^1\times\mathbb C,
\qquad
f(z,w) = \bigl(z^d, q_z(w)\bigr),
\]
where $q_z\defeq q(z,\cdot)$.
For every such $f$,
 the \emph{fiberwise Green function} is defined by
\[
G(z,w) = G_z(w)
\;\defeq\;
\lim_{n\to\infty}\frac{1}{d^n}\log^+|q_z^{\,n}(w)|,
\qquad
\mbox{ where }
\log^+|w|\defeq\max\{\log|w|,0\}.
\]
The function $G$ is continuous on $S^1\times\mathbb C$. For each $z\in S^1$,
 the function $G_z$ is subharmonic, and equals $0$ on the set $K_z$ of points with bounded orbit. It 
is harmonic on the
open set $W_z \defeq \mathbb C \setminus K_z = \{G_z>0\}$.
The \emph{fiberwise Julia set} is $J_z(f)\defeq\partial K_z = \partial W_z$, 
and the Julia set of $f$ is
\[
J(f) \defeq \overline{\bigcup_{z\in S^1} \{z\}\times J_z(f)} \;\subset S^1\times\mathbb C.
\]
We also set $J(F)= J(f)$.

\medskip

We will only consider
fibered polynomial maps for which $K_z$ is connected for every $z\in S^1$.
 Equivalently, 
for every $z\in S^1$,
 all critical points of $q_z$ have bounded orbit.
 We call any $f$ with this property \emph{fiberwise
connected}.
In this case, every $W_z\cup \{\infty\}$ is simply connected and 
  biholomorphic to the unit disc.
   We will denote by $\Delta$
the unit disc in $\mathbb C$ and by $1/\Delta$ the complement of the closed unit disc in $\mathbb C$.
The following result is due to Sester \cite{Sester99}.

\begin{prop}[{\cite[Proposition~2.7]{Sester99}}]\label{prop_sester}
For every 
fiberwise connected
 $f$ as above
and $z\in S^1$, there exists a unique holomorphic isomorphism
$\varphi_z \colon W_z \longrightarrow 1/\Delta$
such that:
\begin{enumerate}
\item the diagram
\[
\begin{tikzcd}
W_z \arrow{r}{q_z} \arrow[swap]{d}{\varphi_z} &
W_{z^d} \arrow{d}{\varphi_{z^d}} \\
1/\Delta \arrow{r}{w^d} &
1/\Delta
\end{tikzcd}
\]
commutes;
\item the map $\varphi\colon W\to\mathbb C$ defined by $\varphi(z,w)=\varphi_z(w)$ is continuous;
\item $G(z,w)=\log|\varphi_z(w)|$ for all $(z,w)\in W$;
\item $\varphi_z$ is tangent to the identity at infinity.
\end{enumerate}
\end{prop}

\begin{rmk}
Let $F(z,w)=(z^d,q(z,w))$ be a holomorphic skew product as in
\eqref{eq:F-skew}.
Then there exist domains $U\Subset V\subset\mathbb C^2$ such that the restriction
$F\colon U\to V$ is a proper holomorphic map.
In particular, $F$ admits a polynomial-like restriction
in the sense of Dinh--Sibony~\cite{DS03}, whose
topological
 degree equals $d^2$,
and the results of \cite{BHMane} apply in this context, see
in particular \cite[Remark 1.4]{BHMane}.

A particularly relevant special case
of holomorphic skew products as in \eqref{eq:F-skew} 
is given by maps for which
each coefficient $c_j(z)$ is a polynomial of degree at most $d-j$.
In this case, $F$ extends to a holomorphic endomorphism of $\mathbb P^2$
(i.e., it is a \emph{regular polynomial skew product} in the sense of \cite{BedfordJonsson, Jonsson99}).
We will denote by $\mathcal P(d)$ 
the space of such polynomial skew products
 endowed with the topology of uniform convergence of coefficients on compact
subsets of $\mathbb C$.
\end{rmk}

\subsection{Uniform hyperbolicity and vertical expansion}\label{sec_expansion}

Following Jonsson~\cite{Jonsson99}, we say that a holomorphic skew product
$F(z,w)=(z^d,q(z,w))$
is \emph{vertically expanding} if there exist a neighbourhood $U$ of the Julia set $J(F)$
and constants $c>0$ and $\lambda>1$ such that
\[
\|D_w F^n(z,w)\|\ge c\,\lambda^n
\quad \text{for all } (z,w)\in U \text{ and all } n\ge1.
\]
In this case, the dynamics is uniformly expanding in the vertical direction near the Julia set.
The following result follows
 from standard arguments in hyperbolic dynamics.
 We refer to \cite{Jonsson99} for several further
equivalent characterizations of vertical expansion.

\begin{prop}\label{prop_vert_exp_hyp}
A skew product $F$ as in \eqref{eq:F-skew} 
is vertically expanding if and only if its Julia set $J(F)$
is a uniformly hyperbolic repeller.
\end{prop}

Vertical expansion is an open condition in $\mathcal S(d)$
(and $\mathcal P(d)$)
 and defines
connected components, which are 
called
\emph{hyperbolic components}.
By \cite{AstorgBianchi1}, these components coincide with
the stability components in the sense of \cite{BBDstab, Bianchi19}
as soon as one parameter inside them is hyperbolic.

For maps belonging to a hyperbolic component, the Julia sets vary continuously with the
map and the dynamics is structurally stable: all maps in the same component are
topologically conjugate on their Julia sets.
Moreover, the conjugacies respect the skew-product structure and depend continuously on
the parameter.

In particular, the skew product
\[
F_0(z,w)=(z^d,w^d)
\]
belongs to a distinguished hyperbolic component of $\mathcal S(d)$
(and $\mathcal P(d)$).
This component plays a role analogous to the
main cardioid of the Mandelbrot set for quadratic polynomials.

\subsection{Conjugacies and regularity inside a hyperbolic component}\label{sec_conjugacies}
Let
 $\mathcal H \subset \mathcal S(d)$ be
  a hyperbolic component
  consisting of fiberwise connected maps,
   and fix a reference map
$F_0\in\mathcal H$.
For any $F\in\mathcal H$, we denote by $J(F)$ the Julia set of $F$, and recall from
Section~\ref{sec_expansion} 
that all maps in $\mathcal H$ are topologically conjugate on their
Julia sets by conjugacies respecting the skew-product structure.
In this subsection, we use 
B\"ottcher coordinates to 
study the regularity of this 
 canonical conjugacy between $J(F_0)$ and $J(F)$.

Fix
 $F\in\mathcal H$.
Since $F$ is fiberwise connected, for every $z\in S^1$ 
the basin of infinity
$W_{F,z}=\mathbb C\setminus K_{F,z}$ admits a B\"ottcher coordinate
\[
\varphi_{F,z}\colon W_{F,z}\longrightarrow 1/\Delta
\]
conjugating $q_{F,z}$ to $w\mapsto w^d$ and tangent to the identity at infinity,
see Proposition~\ref{prop_sester}.
We define a fibered map
\[
H_F\colon W_{F_0}\longrightarrow W_F,
\qquad
H_F(z,w) = \bigl(z, H_{F,z}(w)\bigr),
\]
by setting
\[
H_{F,z} \defeq \varphi_{F,z}^{-1}\circ \varphi_{F_0,z}
\quad \text{on } W_{F_0,z}.
\]
The  following lemma is an immediate consequence of
the defining properties of the B\"ottcher coordinates.

\begin{lem}\label{lem_interior_conjugacy}
The map $H_F$ is well-defined and satisfies:
\begin{enumerate}
\item $H_{F_0}=\mathrm{id}$;
\item for each $z\in S^1$,
 the map $H_{F,z}$ is conformal on $W_{F_0,z}$;
\item $H_F$ conjugates $F_0$ and $F$ on the basin of infinity:
\[
F\circ H_F = H_F\circ F_0 \quad \text{on } W_{F_0}.
\]
\end{enumerate}
\end{lem}

We now consider
 the extension of $H_F$ to the Julia set.
 The basins of infinity
$\{W_{F,z}\}_{z\in S^1}$ form a family of
John domains
\cite{DH,Pommerenke},
uniformly in 
$z\in S^1$ 
(see, for instance, \cite{Sumi06}),
 and locally uniformly in 
$F\in \mathcal H$
(as the maximal expansion $\sup_{J(F)}\|DF\|$ and the minimal expansion
$\inf_{J(F)}\|DF^{-1}\|^{-1}$ depend continuously on $F$).
As a consequence, the B\"ottcher coordinates admit H\"older continuous boundary extensions
with uniform H\"older exponent and constant 
\cite{CJY94}.

\begin{lem}
\label{lem_holder_conjugacy}
There exist $\alpha\in(0,1)$ and $C>0$ such that for every $F\in\mathcal H$,
the conjugacy $H_F$ extends uniquely to a map
\[
\hat H_F\colon W_{F_0}\cup J(F_0)\longrightarrow W_F \cup J(F)
\]
which is $\alpha$-H\"older continuous with H\"older constant bounded by $C$.
Moreover, $\hat H_F$ conjugates the dynamics on the Julia sets and respects the
skew-product structure.
\end{lem}

Since both the extension given by Lemma~\ref{lem_holder_conjugacy} and structural stability provide
conjugacies on $J(F_0)$, these maps coincide on $J(F_0)$ (by uniqueness of the conjugacy inside a
hyperbolic component). In particular, the conjugacies among Julia sets are H\"older continuous,
uniformly in $z$ and locally uniformly in $F$.

We now strengthen
 the 
(local) 
compactness
of the maps $\{\hat H_F\}_{F}$ in the H\"older norm, thanks to the fact that,
for every $(z,w)\in J(F_0)$
 the conjugacy
map $F\mapsto H_F (z,w)$
is actually \emph{holomorphic} in $F$.
Let $G_t, t\in (-\varepsilon,\varepsilon)$ be a parametrization of a $C^2$
path in
$\mathcal H$
 with
$G_0=F_0$.
Recall that, given a path
$(\psi_t)_{ t\in(-\varepsilon,\varepsilon)}$ 
in $C^\alpha\bigl(J(F_0)\bigr)$,
we can define 
\begin{equation}\label{eq:def-dot-ddot-phi}
\dot\psi_t\defeq\left.\frac{d}{ds}\right|_{s=t}\psi_s
\qquad
\mbox{ and }
\qquad
\ddot\psi_t\defeq\left.\frac{d^2}{ds^2}\right|_{s=t}\psi_s.
\end{equation}
We say that
$(\psi_t)_{ t\in(-\varepsilon,\varepsilon)}$ is $C^1$
 (resp.\ $C^2$) if $\dot\psi_t$ (resp. $\ddot\psi_t$)
 defines a continuous path in 
 $C^\alpha\bigl(J(F_0)\bigr)$.

\begin{prop}
\label{prop_param_regularity}
The map
$t\mapsto \hat H_{G_t}$
is $C^2$ 
as a map from $(-\varepsilon,\varepsilon)$ into the H\"older space
$C^\alpha\bigl(W_{F_0} \cup J(F_0)\bigr)$.
\end{prop}

\begin{proof}
For every given $(z,w)$,
the map
$F\mapsto H_F(z,w)$ is
 constructed as a uniform limit
of holomorphic functions.
 Hence, it
depends holomorphically on  $F$.
It is 
standard
that the uniform
 control 
given by Lemma \ref{lem_holder_conjugacy} 
allows one to apply Cauchy estimates to obtain
uniform bounds on derivatives, yielding smoothness
 in the H\"older norm, see, for instance, \cite{SU10}
 for similar arguments.
\end{proof}

\section{Proof of Theorem \ref{thm_main_sk}}
\label{s:proof-thm}

\subsection{From $\ddot \delta_0$ to the variance of $\dot \phi_0$}\label{sec_derivative}

We denote by
 $F_0$ the map 
$(z,w)\mapsto (z^d,w^d)$
and by $m_0 \defeq \Leb_{1}\times \Leb_{1}$ its unique measure of maximal entropy.
 Let $\mathcal H\subset\mathcal S(d)$ be the
 hyperbolic component of $\mathcal S(d)$ containing $F_0$.
Consider a $C^2$ curve $(\varepsilon, \varepsilon)\ni t\mapsto F_t\in\mathcal H$ with $F_{t=0}=F_0$.
By Section~\ref{sec_conjugacies}, for each $t$
 there exists a conjugacy
\[
\hat H_t\colon J(F_0)\longrightarrow J(F_t)
\]
which is H\"older continuous in $(z,w)$
 and depends $C^2$ on $t$ with respect to the 
  H\"older topology of $C^\alpha(J(F_0))$ for some $\alpha\in (0,1)$.

\medskip

Consider the family of geometric 
potentials on $J(F_0)=S^1 \times S^1$ 
given by
\[
\phi_t \defeq -\log \bigl|\Jac F_t \circ \hat H_t\bigr|.
\]
By Proposition~\ref{prop_param_regularity} and the smooth dependence of
$(z,w)\mapsto \log|\Jac F_t(z,w)|$ on $t$, the path $\{\phi_t\}_t$
is $C^2$ in $C^\alpha(J(F_0))$ for some $\alpha\in (0,1)$. 
In particular, $\dot \phi_0$ and $\ddot \phi_0$
 are well-defined 
 by \eqref{eq:def-dot-ddot-phi} 
 and H\"older continuous on $J(F_0)= S^1 \times S^1$.
 
 \begin{lem}\label{l:phi}
 We have $\int \dot \phi_0\, dm_0=0$ and  $\int \ddot \phi_0\, dm_0=0$
 \end{lem}
 
 \begin{proof}
 Observe that the 
 function $t\mapsto L(t)\defeq -\int \phi_t dm_0$ is the sum of the Lyapunov exponents of 
 $F_t$. The assertion follows from the fact that
  this function is constant
  (and equal to $2\log d$)
   in a neighbourhood of $F_0$, see
   \cite[Theorem 5.3]{Jonsson99}.
 \end{proof}
\medskip

Let $\delta_t$ denote the unique positive real number such that
\[
P(\delta_t \phi_t)=0,
\]
where $P(\cdot)$ denotes the topological pressure.
By real-analyticity of the pressure for H\"older potentials and the implicit function theorem,
the map $t\mapsto\delta_t$ is real-analytic on a neighbourhood of $0$
(see \cite{Ruelle82}).

\begin{lem}\label{l:delta}
We have $\delta_0=1$ and $\dot \delta_0=0$.
\end{lem}

\begin{proof}
We have $\phi_0 \equiv -2\log d$ on $J(F_0) = S^1\times S^1$.
Therefore, we have
 $P(s\phi_0)=\log(d^2)-s\log(d^2)$ for every $s\in \mathbb R$.
  Hence, the unique solution
  to $P(s\phi_0)=0$ 
  is $s=1$, which proves the first assertion.

\medskip

To prove the
 second one, differentiate the identity $P(\delta_t\phi_t)=0$ at $t=0$.
By standard arguments
in thermodynamic formalism 
(see
e.g., \cite{PU}), we have
\[
0=\left.\frac{d}{dt}\right|_{t=0}P(\delta_t\phi_t)
=\int (\dot\delta_0\,\phi_0+\delta_0\dot\phi_0)\,dm_0,
\]
since $m_0$ is the equilibrium state of $\delta_0\phi_0=\phi_0\equiv -2\log d$.
Since $\phi_0$ is constant and $\delta_0=1$, this gives
\[
0=\dot\delta_0\int\phi_0\,dm_0+\int\dot\phi_0\,dm_0.
\]
Since $\int\dot\phi_0\,dm_0=0$
by Lemma \ref{l:phi},
the assertion follows.
\end{proof}

  Recall that,
given 
a continuous function $\psi\colon J(F_0)\to\mathbb R$ 
with $\int \psi dm_0=0$,
the (asymptotic)
\emph{variance} of $\psi$ with respect to $m_0$
 is defined by
\begin{align*}\label{eq_def_var}
{\rm Var}(\psi,m_0)
&\defeq
\lim_{n\to\infty}\frac1n
\int_{J(F_0)}\left(\sum_{j=0}^{n-1}\psi\circ F_0^j\right)^2\,dm_0
=
\lim_{n\to\infty}\frac1n\|S_n\psi\|_{L^2(m_0)}^2
\;\in[0,+\infty].
\end{align*}

\begin{prop}\label{thm_delta_second_derivative}
We have
\[
\ddot \delta_0 
=
\frac{
{\rm Var}\bigl(\dot\phi_0,m_0\bigr)}{2\log d}.
\]
\end{prop}

\begin{proof}
Set $\psi_t\defeq\delta_t\phi_t$.
By the regularity established above, $\{\psi_t\}$ is a $C^2$ path 
in
$C^\alpha(J(F_0))$.
As $P(\psi_t)\equiv 0$, 
it follows from 
\cite[Theorem~2.2]{McMullen08}
(see also \cite{PP90}) that
\[
0 = 
{\rm Var}
(\dot \psi_t, m_0) 
+\int \ddot \psi_0\,dm_0.
\]
A direct development using $\dot\psi_0 = \delta_0 \dot \phi_0 + \dot \delta_0 \phi_0$
and $\ddot \psi_0 =  \delta_0 \ddot \phi_0 + 2\dot \delta_0 \dot \phi_0  + \ddot \delta_0 \phi_0$
gives
\[
\ddot\delta_0
=
\frac{
{\rm Var}\bigl(\delta_0\dot\phi_0+\phi_0\dot\delta_0,m_0\bigr)
+\delta_0\int \ddot\phi_0\,dm_0
+2\int \dot\delta_0\dot\phi_0\,dm_0
}{
-\int\phi_0\,dm_0
}.
\]
This identity reduces to the one in the statement thanks to Lemmas
 \ref{l:phi} and  \ref{l:delta} and the fact that $\phi_0 \equiv -2\log d$ on $S^1\times S^1$.
\end{proof}

\subsection{Variance of $\dot \phi_0$ as an asymptotic energy of $\partial_w v$}\label{sec_CM2}

In order to
 better describe 
$\dot\phi_0$, we consider the infinitesimal
deformation of the conjugacies on the basin of infinity.
Let
\[
\widetilde v=\left.\frac{d}{dt}\right|_{t=0}\hat H_t
\]
denote the derivative of the conjugation
 given by
the B\"ottcher coordinates, see Section \ref{sec_conjugacies}.
Since the conjugacies are fibered, $\widetilde v$ is vertical and can be written as
\begin{equation}\label{eq:v}
\widetilde v(z,w)=v(z,w)\,\frac{\partial}{\partial w}
\end{equation}
for every $z\in S^1$,
where $v(z,\cdot)$ is holomorphic on the basin of infinity $W_{F_0,z}=1/\Delta$. 
Moreover, $\widetilde v$ is smooth in $w$
on $W_{F_0}$.

\begin{lem}\label{lem_coboundary_identity}
We have
\[
\dot\phi_0(z,w)
=
-\Re\!\left(\frac{\partial v}{\partial w}(z,w)\right)
+
\Re\!\left(\frac{\partial v}{\partial w}\bigl(F_0(z,w)\bigr)\right).
\]
\end{lem}

\begin{proof}
From $F_t\circ H_t=H_t\circ F_0$ we obtain
\[
\Jac F_t(H_t)= (\Jac H_t)^{-1}\,(\Jac H_t\circ F_0)\,\Jac F_0.
\]
Taking logarithms and differentiating at $t=0$, using $\Jac H_0\equiv1$ 
yields the stated identity.
\end{proof}

\begin{prop} \label{thm_main_2}
With the above notation,
 we have
\[{\rm Var}(\dot\phi_0,m_0) = 
\frac{\log d}{2}
 \cdot I
\left(
\frac{\partial v}{\partial w}\right)
\]
where
\[
I
\left(
\frac{\partial v}{\partial w}\right) \defeq
\lim_{r \to 1^+} \frac{1}{|\log(r-1)|}
 \int_{S^1}
\int_{|w|=r}\left|\frac{\partial v}{\partial w}(z,w)\right|^2
d\Leb_r (w)
d\Leb_1(z)
\]
and $\Leb_r$ is the normalized Lebesgue
probability
 measure on the circle $|w|=r$.
\end{prop}

\begin{proof}
For simplicity, we will denote $g \defeq\Re \left(\frac{\partial v}{\partial w}\right)= \Re ( \partial_w v)$,
 and observe that
we have $\dot \phi_0 =- g + g\circ F_0$ on $W_{F_0} =S^1\times (1/\Delta)$.
Since $\partial_w v(z,\cdot)$ is holomorphic on $\{|w|>1\}$ and vanishes at
infinity, we have
\[
\int_{|w|=r}
\left|\Re
\left(\frac{\partial v}{\partial w} (z,w)\right)
\right|^2\,d\Leb_r(w)
=\frac12\int_{|w|=r}\left|
\left(\frac{\partial v}{\partial w} (z,w)\right)
\right|^2\,d\Leb_r (w)
\]
for every $r>1$ and  $z\in S^1$.
 Therefore,
it is enough to show the
identity
\[
{\rm Var}(\dot\phi_0,m_0)
=
(\log d) \cdot
\lim_{r \to 1^+} \frac{1}{|\log(r-1)|}
 \int_{S^1}\int_{|w|=r}
|g|^2
d\Leb_r 
(w)
d\Leb_1 (z).
\]

We will show the assertion for a specific choice of $r_n\to 1$, the argument can be easily
adapted
to handle the general limit.
Recall that we have 
\[
{\rm Var} (\dot\phi_0,m_0) = \lim_{n\to \infty}
\frac{1}{n}
\int_{J(F_0)}
\left|
 \sum_{j=0}^{n-1} \dot \phi_0 \circ F_0^j\right|^2\,d m_0.
\]
Set
 $r_n\defeq e^{d^{-n}}\sim 1+ d^{-n}$
 and $Z_n(z,w)\defeq(z,r_n w)\in W_{F_0}$.
 Using the fact that $Z_n(z,w)\to(z,w)$ exponentially
  fast and
  $\dot \phi_0$ is H\"older continuous
  in a neighbourhood of $S^1\times S^1$,
   for every
 $(z,w)\in S^1\times S^1$
 we obtain
\[
 \sum_{j=0}^{n-1}
 \dot \phi_0 \circ F_0^j (z,w)
 =
 \sum_{j=0}^{n-1}
 \dot \phi_0 \circ F^j_0 (Z_n (z,w))
  + O(1), \qquad n\to \infty.
\]
We obtain
from Lemma \ref{lem_coboundary_identity} that
\[
\sum_{j=0}^{n-1}
\dot \phi_0 \circ F_0^j ( Z_n (z,w))
=-g(Z_n(z,w))
+
g\circ F_0^{n} (Z_n(z,w))
=-g(Z_n(z,w))
+O(1), \qquad n\to \infty,
\]
where 
 $g\circ F^{n}_0 (Z_n(z,w))$
 is uniformly bounded as all the points 
 $F^{n}_0 (Z_n(z,w))$ belong to the circle of radius
  $(e^{d^{-n}})^{d^n}=e$.
  It follows that
\[\begin{aligned}
{\rm Var} (\dot\phi_0,m_0) & =  \lim_{n\to \infty}
\frac{1}{n}
\int_{J(F_0)}
|g(Z_n(z,w))
|^2 
dm_0
\\ &=
 \lim_{n\to \infty}
\frac{1}{n}
\int_{S^1}
\int_{|w|=r_n}
|g (z,w)
|^2
d\Leb_{r_n} (w)
 d\Leb_1(z).
\end{aligned}
\]
The assertion follows taking into account the asymptotic
$|\log(r_n-1)|\sim n\log d$.
\end{proof}

\begin{rmk}
Following the terminology of \cite{McMullen08}, we can
 say that 
$\dot \phi_0$ is the \emph{virtual coboundary} of 
$\Re (\partial_w v)$
on $W_{F_0}=S^1\times (1/\Delta)$.
The above arguments show more generally that, if 
 $h$ is a
 virtual coboundary of a continuous function $g$
 on $W_{F_0}$ with
$\int h\,dm_0=0$,
then
\[
{\rm Var}(h,m_0)=(\log d)\cdot I(g),
\]
where
\[
I(g)=
\lim_{r\to 1^+}
\frac{1}{|\log(r-1)|}
\int_{S^1}
\int_{|w|=r}|g(z,w)|^2\,
d\Leb_{r} (w)
\,
d\Leb_1 (z).
\]
\end{rmk}

\subsection{A computation of $I(\partial_w v)$}
\label{sec_CM5}

We
now consider the explicit path
\[
F_t(z,w)=(z^d,\,w^d+t(c_1(z)w^{d-1}+\cdots+c_d(z))),
\qquad t\in(-\varepsilon,\varepsilon).
\]

By Proposition
\ref{thm_main_2},
 the term
${\rm Var}(\delta_0\dot\phi_0,m_0)$ 
appearing in 
Proposition \ref{thm_delta_second_derivative}
can be expressed in terms of the asymptotic energy
$I(\partial_w v)$ of the infinitesimal conjugacy.
The next proposition gives an explicit expression 
for $I(\partial_w v)$
 in terms of
the functions $c_k$.

\begin{prop}\label{prop_extra_term}
We have
\[
I
\left(
\frac{\partial v}{\partial w}\right)
=
\frac{1}{d^2\log d
}
\int_{S^1}
\sum_{k=1}^d k^2\,|c_k(z)|^2\,
d\Leb_1 (z).
\]
\end{prop}

 For simplicity,
we will set $v_z(\cdot) = v(z,\cdot)$, where
$v(z,\cdot)$ is as in \eqref{eq:v}.

\begin{lem}\label{lem_expansion_v}
For each $z\in S^1$,
 we have
\begin{equation}\label{eq_v_outside}
v_z(w)
=
-\frac{w}{d}
\sum_{k=1}^d\sum_{n=0}^{\infty}
\frac{c_k(z^{d^n})}{d^n}\,w^{-k d^n}.
\end{equation}
\end{lem}

Observe that the series above 
converges normally on $\{|w|>1+\eta\}$ for every $\eta>0$.

\begin{proof}
Differentiating the conjugacy relation
$F_t\circ H_t=H_t\circ F_0$ at $t=0$,
we obtain that,
for every $z\in S^1$,
 $v_z$ 
is the solution of the functional equation
\[
v_{z^d}(w^d)
=
d\,w^{d-1}v_z(w)
+
\sum_{k=1}^d c_k(z)w^{d-k}
\]
satisfying
$v_z(w)\to0$ as $w\to\infty$.
A direct computation shows that  
the series~\eqref{eq_v_outside}
satisfies these conditions.
\end{proof}

\begin{proof}[Proof of Proposition \ref{prop_extra_term}]
Differentiating~\eqref{eq_v_outside}
 term by term gives
\[
\frac{\partial v}{\partial w}(z,w)
=
\frac{1}{d}
\sum_{k=1}^d\sum_{n=0}^{\infty}
\left(
\frac{c_k(z^{d^n})}{d^n}\,w^{-k d^n}
+
k\,c_k(z^{d^n})\,w^{-k d^n-1}
\right).
\]
The terms involving the denominator $d^{n}
$ are uniformly summable and negligible in the
logarithmic limit defining $I(\partial_w v)$.
Hence only the second sum contributes, and we obtain
\[
I\!\left(\frac{\partial v}{\partial w}\right)
=
\lim_{r\to1^+}
\frac{1}{|\log(r-1)|}
\int_{S^1}
\int_{|w|=r}
\left|
\frac{1}{d}
\sum_{k=1}^d
\sum_{n=0}^{\infty}
k\,c_k(z^{d^n})\,w^{-k d^n-1}
\right|^2
\, 
d\Leb_r(w)\,
d\Leb_1(z).
\]

Since all exponents $-k d^n-1$ are distinct, the functions
$w\mapsto w^{-k d^n-1}$
 are orthogonal in $L^2(\Leb_r)$.
Therefore, for every $z\in S^1$, we have
\[
\int_{|w|=r}
\left|
\frac{1}{d}
\sum_{k=1}^d
\sum_{n=0}^{\infty}
k\,c_k(
z^{d^n}
)\,w^{-k d^n-1}
\right|^2
\,
d\Leb_r (w)
=
\frac{1}{d^2}
\sum_{k=1}^d\sum_{n=0}^{\infty}
k^2|c_k(z^{d^n})|^2\,|r|^{-2k d^n-2}.
\]
For any fixed $r>1$,
 the factor $|r|^{-2k d^n-2}$
  is close to $1$ for
$n\le N(r)$ and exponentially small for $n>N(r)$, where
\[
N(r)\sim\frac{|\log(r-1)|}{\log d}.
\]
Hence,
\[
\int_{|w|=r}
\left|
\frac{1}{d}
\sum_{k=1}^d
\sum_{n=0}^{\infty}
k\,c_k(z^{d^n})\,w^{-k d^n-1}
\right|^2
\,
d\Leb_r (w)
=
\frac{1}{d^2}
\sum_{k=1}^d\sum_{n=0}^{N(r)}
k^2|c_k(z^{d^n})|^2
+
o(|\log(r-1)|).
\]

Dividing by $|\log(r-1)|$ and letting $r\to1$, we may apply the Birkhoff
ergodic theorem to the sequence $\{z^{d^n}\}_{n\in \mathbb N}$.
For 
$\Leb_1$-almost every $z\in S^1$
 we obtain
\[
\lim_{r\to1^+}
\frac{1}{|\log(r-1)|}
\int_{|w|=r}
\left|\frac{\partial v}{\partial w}(z,w)\right|^2\,
d\Leb_r (w)
=
\frac{1}{d^2\log d}
\sum_{k=1}^d
k^2\int_{S^1}|c_k(z')|^2\,
d\Leb_1(z')
\]

Integrating with respect to $z$
(observe that the right hand side is constant)
 yields the assertion.
\end{proof}

\subsection{End of the proof of Theorem \ref{thm_main_sk}}

We can now conclude the proof of Theorem \ref{thm_main_sk}.
By \cite[Theorem~1.3]{BHMane}, for every 
 $t\in(-\varepsilon,\varepsilon)$ we have $2\,\VD_{f_t}(J_t)=\delta_t$. The assertion then follows
  from Lemma \ref{l:delta} and
 Propositions \ref{thm_delta_second_derivative}, \ref{thm_main_2}, and \ref{prop_extra_term}.

\begin{rmk}
We observe that the same proof applies when the first component of $F_t$
is $z^{d'}$.
In this case, the factor $2\log d$ appearing in the denominator of the
coefficient in Theorem~\ref{thm_main_sk} is replaced by $\log d+\log d'$.
This change comes from the denominator in the second-derivative formula
of Proposition~\ref{thm_delta_second_derivative}, where
$\phi_0=-\log|\Jac F_0|=-(\log d+\log d')$.
\end{rmk}

\printbibliography
\end{document}